\documentclass[12pt]{article}
\usepackage{geometry}
\usepackage{amsmath, amsfonts,amssymb,mathrsfs,amscd,comment}
\usepackage{comment}
\usepackage{color}
\usepackage[mathscr]{eucal}
\usepackage[math]{easyeqn}
\usepackage{etoolbox}
\usepackage{amssymb}
\usepackage[table]{xcolor}
\usepackage{pgfplots}
\usepackage{multirow}
\usepackage{listings}
\usepackage{emptypage}
\usepackage{natbib}
\usepackage{chngcntr}
\usepackage{sidecap}
\usepackage{complexity}
\usepackage{caption}
\usepackage{hyperref}
\definecolor{darkgreen}{rgb}{0,0.6,0}
\definecolor{darkblue}{rgb}{0,0,0.6}
\definecolor{orange}{RGB}{119,34,51}
\hypersetup{linktocpage,
            colorlinks,
            linkcolor=darkgreen,
            linktoc=blue,
            citecolor=darkblue,
            urlcolor=black,
            filecolor=black
            }

\usepackage{nameref}
\usepackage{url}

\usepackage{amsthm}

\newtheorem{theorem}{Theorem}[section]
\numberwithin{equation}{section}
\newcounter{example}[section]
\numberwithin{example}{section}
\newtheorem{proposition}[theorem]{Proposition}

\newtheorem{lemma}[theorem]{Lemma}
\newtheorem{assumption}[theorem]{Assumption}
\newtheorem{corollary}[theorem]{Corollary}
\newtheorem{definition}[theorem]{Definition}
\newtheorem{exmp}[example]{Example}

\newtheorem{remark}[theorem]{Remark}
\newenvironment{example}{\begin{exmp}\rm}{\end{exmp}}

\def\cond{\, \big| \,}

\def\Ind{\operatorname{1}\hspace{-4.3pt}\operatorname{I}}

\def\Ph{\hat{P}}
\def\Ch{\hat{C}}

\def\ex{\mathrm{e}}

\usepackage{amsthm}
\usepackage{thmtools}
\declaretheoremstyle[headfont=\scshape]{normalhead}

\def\Pb{{\mathbf P}}

\def\eps{\epsilon}

\def\Ac{A_{0}}
\def\Ab{A_{\rd}}

\def\Cb{\cc{C}_{\rdb}}

\def\Ac{A_{0}}

\def\Ab{A_{\rdb}}

\def\Cb{\cc{C}_{\rdb}}

\def\N{\mathcal{N}}

\def\Ec{\E^{\circ}}

\def\F{\mathcal{F}}

\def\P{{\mathbb P}}
\def\E{{\mathbb E}}

\def\cond{\, \big| \,}

\def\Ind{\operatorname{1}\hspace{-4.3pt}\operatorname{I}}

\def\ex{\mathrm{e}}

\def\eps{\epsilon}

\def\Ac{A_{0}}
\def\Ab{A_{\rd}}

\def\Cb{\cc{C}_{\rdb}}

\def\Ac{A_{0}}

\def\Ab{A_{\rdb}}

\def\Cb{\cc{C}_{\rdb}}

\def\N{\mathcal{N}}

\def\Ec{\E^{\circ}}

\renewcommand{\(}{$\,}
\renewcommand{\)}{\,$}

\def\eqdef{\stackrel{\operatorname{def}}{=}}

\renewcommand{\bar}[1]{\overline{#1}}
\renewcommand{\hat}[1]{\widehat{#1}}
\renewcommand{\tilde}[1]{\widetilde{#1}}

\renewcommand{\Gamma}{\varGamma}
\renewcommand{\Pi}{\varPi}
\renewcommand{\Sigma}{\varSigma}
\renewcommand{\Delta}{\varDelta}
\renewcommand{\Lambda}{\varLambda}
\renewcommand{\Psi}{\varPsi}
\renewcommand{\Phi}{\varPhi}
\renewcommand{\Theta}{\varTheta}
\renewcommand{\Omega}{\varOmega}
\renewcommand{\Xi}{\varXi}
\renewcommand{\Upsilon}{\varUpsilon}

\def\Var{\operatorname{Var}}

\def\argmin{\operatornamewithlimits{argmin}}

\def\R{I\!\!R}

\def\kappa{\varkappa}

\def\hv{\bb{h}}

\def\d{{d}}

\def\Kb{\bar{\Kappa}}

\def\Kb{\bar{\Kappa}}

\def\Ind{{\bf 1}}

\def\Ch{\hat{C}}

\def\R{\mathbb{R}}

\def\Kh{\mathcal{K}}

\def\P{\mathbb{P}}
\def\R{\mathbb{R}}
\def\E{\mathbb{E}}

\def\Nc{\mathcal{N}}
\def\Ch{\hat{C}}
\def\Var{{\rm Var}}
\def\bigO{\mathcal{O}}
\def\Ah{\hat{A}}

\def\Ab{{\mathbf A}}
\def\hv{\hat\vartheta}
\def\hvo{\hv_{\text{oracle}}}
\def\eps{\epsilon}

\def\Ind{{\bf 1}}
\def\d{{\rm d}}
\def\e{{\rm e}}
\def\ex{{\rm e}}

\def\F{\mathcal{F}}
\def\Eb{{\mathbf E}}
\def\Ec{\mathcal{E}}
\def\Ac{\mathcal{A}}
\def\Bf{\mathfrak{B}}
\def\Kb{{\mathbf K}}
\def\Kh{\mathcal{ K}}
\def\Cb{{\mathbf C}}

\def\Sb{{\mathbf S}}

\def\cond{\, \big| \,}
\def\Lc{L}

\def\N{{\mathbb N}}
\def\wrap{\text{wrap}_\Ab}

\def\Var{\text{Var}}

\def\Ec{\mathcal{E}}

\numberwithin{equation}{section}

\swapnumbers

\RequirePackage[normalem]{ulem}
\RequirePackage{color}\definecolor{RED}{rgb}{1,0,0}\definecolor{BLUE}{rgb}{0,0,1}

\renewcommand{\d}{\ensuremath {\,\text{d}}}
\renewcommand{\eps}{\varepsilon}
\renewcommand{\phi}{\varphi}

\renewcommand{\subset}{\subseteq}

\renewcommand{\cdot}{{\scriptstyle \bullet} }

\renewcommand{\abs}[1]{\lvert #1 \rvert}

\renewcommand{\le}{\leqslant}
\renewcommand{\ge}{\geqslant}

\makeatletter
\newsavebox{\@brx}
\newcommand{\llangle}[1][]{\savebox{\@brx}{\(\m@th{#1\langle}\)}%
  \mathopen{\copy\@brx\mkern2mu\kern-0.9\wd\@brx\usebox{\@brx}}}
\newcommand{\rrangle}[1][]{\savebox{\@brx}{\(\m@th{#1\rangle}\)}%
  \mathclose{\copy\@brx\mkern2mu\kern-0.9\wd\@brx\usebox{\@brx}}}
\makeatother

\allowdisplaybreaks[1]
\begin{document}



\title{The wrapping hull and a unified framework\\ for volume estimation}
\author{ \parbox{7cm}{\centering Nicolai Baldin\footnote{Financial support by the European Research Council (ERC) Grant
No. 647812 is particularly acknowledged.
}\\\small Statslab, Department of Pure Mathematics and Mathematical Statistics\\ University of Cambridge\\n.baldin@statslab.cam.ac.uk}  }




\maketitle

\begin{abstract}
This  paper develops a unified framework for estimating the volume of a set in $\mathbb{R}^d$ 
based on observations of points uniformly distributed over
the set. 
The framework applies to all classes of sets satisfying one simple axiom: 
a class is assumed to be intersection stable. No further hypotheses on the boundary of the set are imposed;
 in particular, the convex sets and the so-called weakly-convex sets 
are covered by the framework. 
The approach rests upon a homogeneous Poisson point process model. 
We introduce the so-called wrapping hull, a generalization of the convex hull,  and prove that it is a sufficient and complete statistic.
The proposed estimator of the volume is simply the volume of the wrapping hull scaled with an appropriate factor.
It is shown to be consistent for all classes of sets satisfying the axiom and 	 mimics an unbiased 
estimator with uniformly minimal variance. 
The construction and proofs hinge upon an interplay between probabilistic and geometric arguments.
The tractability of the framework is numerically confirmed in a variety of examples.

\end{abstract}







\section{Introduction}
The problem of estimating the support of a density has received a large amount of attention in the statistical literature since the
1980s
partly because of vast areas of applications in image analysis, signal processing and econometrics. The fundamental results in 
this area were obtained in \cite{devroye1980detection, KorTsy1993, korostelev1993minimax, korostelev1994asymptotic, cuevas1997, CholaquidisFraiman}. Furthermore, see \cite{tsybakov1997nonparametric,walther1997, rigollet2009optimal,mason2009, Cholaquidis17} for a more general problem of estimating the level sets of a density. 
In particular, 
\cite{KorTsy1993} established the minimax optimal rates for estimating the support of a density having a H{\"o}lder-continuous boundary in the Hausdorff and 
symmetric difference metrics and constructed 
an estimator which attains the optimal rates. The case of convex support estimation was first studied in \cite{korostelev1994asymptotic,korostelev1993minimax}, where it was shown 
that the convex hull \(\Ch\) of the sample points, which is the maximum likelihood estimator, is rate-optimal for estimating the support set \(C\) in the Hausdorff and 
symmetric difference metrics.

The volume of a set is one of its most basic characteristics. Surprisingly, as it was shown in \cite{KorTsy1993, korostelev1993minimax},
the volume of a rate-optimal estimator of the set is not necessarily a rate-optimal  
estimator of the volume. The first fully rate-optimal estimators of the volume of a convex support with smooth boundary and a support 
with H{\"o}lder-continuous boundary  were 
constructed by \cite{gayraud1997} based on three-fold sample splitting. A more efficient and flexible estimator
of the volume of a convex set with no assumption on the boundary was recently proposed in \cite{BaldinReiss16}.
In fact the proposed estimator is simple to compute and non-asymptotically efficient. 
The problem of calculating the volume of a convex set has also attracted mathematicians working in 
computer science and computational geometry, see \cite{dyer1991random,KanLovSim1997, lovasz2006simulated}. 
The setting is slightly different: an experimenter uses a sampling algorithm that generates points over the space 
that either fall inside or outside a set (this information is obtained via queries to the oracle).
We refer to \cite{vempala2010recent} for a recent survey of the existing fast randomised algorithms for calculating the volume of a convex set. 
To the best of our knowledge, no efficient estimators 
of the volume of a more general than convex class of sets have been rigorously studied.


\subsection{Main contribution and the structure}

In the present paper, we combine techniques from statistics and stochastic geometry to build a general framework for estimating
the volume of a set. 
We focus on the Poisson point process (PPP) observation model with intensity $\lambda>0$ on a set $A$. 
We thus observe
\begin{EQA}[c]
	X_1, ..., X_N \overset{i.i.d.}{\sim} U(A), \, \quad
	N \sim \text{Poiss}(\lambda\abs{A}),
\end{EQA}
where \((X_n), N\) are independent and \(|A|\) denotes the volume or the Lebesgue measure of the set \(A\). The set \(A\) is meant to belong to a class \(\Ab\) 
satisfying one simple assumption: \emph{the class is assumed to be intersection stable}, 
see Section~\ref{PPPTheory} for a concise definition.
The classes of sets covered by the assumption hence include:
\begin{itemize}
	\item convex sets;
	\item weakly-convex sets;
	\item star-shaped sets with a H{\"o}lder-continuous boundary;
	\item concentric sets;
	\item polytopes with fixed directions of outer unit normal vectors;
	\item compact sets.
\end{itemize}
In Section~\ref{PPPTheory},  we introduce
the so-called \emph{wrapping hull} \(\Ah\), which can informally be described as the minimal set from the class that 
contains the data points \(X_1,...,X_N\).
It is then used in Section~\ref{OracleCaseSection} to construct the so-called oracle estimator for
the volume of a set belonging to one of the aforementioned classes when the intensity \(\lambda\) of the process  is known.
The oracle estimator is shown to be uniformly
of minimum variance among unbiased estimators (UMVU). 
Section~\ref{DataDriveEstimation} is devoted to the estimation of the intensity and derives 
a fully data-driven estimator of the volume,
\begin{EQA}[c]
	\label{FinEst0}
	\hat\vartheta = \frac{N + 1}{N_\circ + 1} |\Ah|\,,
\end{EQA}
where \(N_\circ\) is the number of data points lying in the interior of the wrapping hull. Figure~\ref{FigureIntro} illustrates an example in which a naive estimator \(|\Ah|\) significantly underestimates the true volume \(|A|\)  even in the case when the class of sets is known whereas the estimator \(	\hat\vartheta\) produces a rather striking performance, see Section~\ref{SimulationsSection} for a more detailed  numerical study\footnote{The simulations were implemented using the R package ``spatstat'' by \cite{spatstat}.}.
\begin{figure}[tp]
	\centering
	\includegraphics[width=\linewidth]{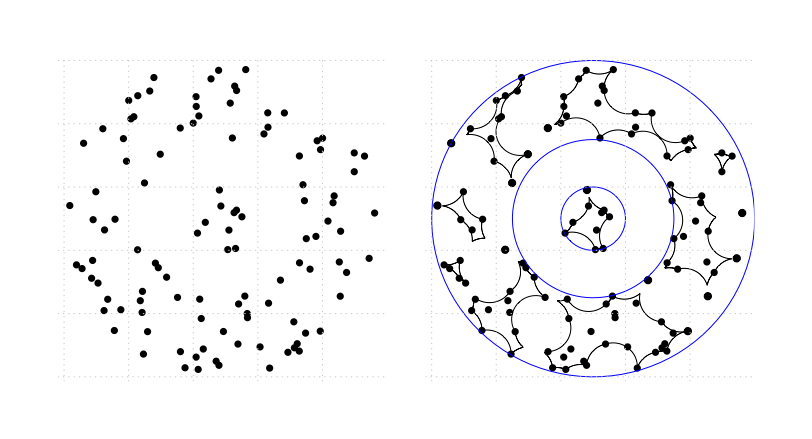}
	\caption[]{On the left:  observations over an $r$-convex set \(A\), the annulus \(B(0.5,0.5)\setminus B(0.5,0.25)\) with the ball inside \( B(0.5,0.1)\). On the right: the $r$-convex hull \(\Ah\) (black) and the true set \(A\) (blue). The volume of the wrapping hull here is \( |\Ah| = 0.307\), the true volume is \(|A| = 0.620\) and our estimator yields \(	\hat\vartheta = 0.578\).}
	\label{FigureIntro}
\end{figure} 
The mean squared risk of the estimator is shown to mimic the mean squared risk of the oracle estimator. 
Although the main object of analysis is the PPP model, the key results 
transfer to the so-called uniform model, cf. Section~\ref{UnifModelVolEst}, using ``Poissonisation''.
Section~\ref{ClassesOfBodies} further establishes the rates 
of convergence of the oracle estimator and the estimator \(\hv\) in \eqref{FinEst0} for the considered classes of sets satisfying the assumption. 
Theorem~\ref{EfronsProp} states a generalized Efron's inequality  for the wrapping hull, cf. \cite{efron1965convex}, which reduces the analysis of the mean squared error of the 
estimator \(\hv\) to the distributional characteristics of the missing volume \(|A \setminus \Ah|\), a uniform lower bound on its expectation and a uniform
deviation inequality. Interestingly, a uniform lower bound on the expectation of the missing volume has not even been established
for the class of convex sets. 
We therefore establish the rates of convergence only for a relatively simple class of polytopes with fixed directions of outer unit normal vectors in Section~\ref{PolFixeDirSec}. A more 
general question is beyond the scope of the present paper and left to future research. 
In volume estimation of weakly-convex sets in Section~\ref{RConnSect} there is a further peculiar question of adaptation 
to a smoothing parameter. 
We suggest an adaptation procedure inspired by Lepski's method, cf. \cite{lepskii1992asymptotically}, and study it numerically in Section~\ref{SimulationsSection}.
Our 
numerical results in Section~\ref{SimulationsSection}, mainly devoted to volume estimation for the weakly-convex sets, in particular, 
demonstrate that overestimating the smoothing parameter may have a significant cost for volume estimation.
Some of the technical lemmata are deferred to the Appendix. 
Finally, we encounter and state a variety of new open questions in stochastic geometry, which we barely 
begin to nibble at the edges. Interestingly enough, the framework was mentioned in a seminal paper \cite{kendall1974foundations} by David Kendall in the Statslab 
at the University of Cambridge, but has never been fully explored. 

\subsubsection{A simple one-dimensional example}
Let  \(X_1, ...,X_n\) be a sample of i.i.d. points drawn from the uniform distribution \(U(a,b)\), and 
let \(X_{(1)},...,X_{(n)}\) denote the order statistics, so that 
\(X_{(1)} < ... < X_{(n)}\).
It holds by symmetry that  the expected length of the interval \((X_{(1)}, X_{(n)})\) satisfies
\begin{EQA}[c]
	\label{EIAX}
	\E [|X_{(n)} - X_{(1)} |] =\frac{(n- 1)}{(n+1)} (b-a)\,.
\end{EQA}
An objective of statistical inference is to estimate the length of the interval, when  the location of the points \(a\) and \(b\) is assumed to be unknown.
A naive estimator, 
\begin{EQA}[c]
	\hat{l}_{naive} := X_{(n)} - X_{(1)}\,,
\end{EQA} clearly 
underestimates  the  length. A more attractive idea is to somehow dilate the interval \((X_{(1)}, X_{(n)})\) 
and take the length of the dilated interval as an estimator. 
There are at least two viable dilations: 1) add and subtract some fixed vectors 
from the end points \(X_{(n)}\) and \(X_{(1)}\) (additive dilation) and 2) dilate the interval \((X_{(1)},X_{(n)})\) 
from its centre \((X_{(n)} + X_{(1)}) / 2\) with some scaling factor (multiplicative dilation). 
In the one-dimensional case, both dilations are equivalent.
It follows from \eqref{EIAX} that a reasonable additive dilation factor is \(2(X_{(n)} - X_{(1)})/(n-1)\) which yields  an estimator for the volume,
\begin{EQA}[c]
	\label{onedim}
	\hat{l}_{1} := \frac{(n+1)}{(n-1)}(X_{(n)} - X_{(1)})\,.
\end{EQA} 

This estimator is not only \emph{unbiased},  \(\E[\hat{l}_{1}] = b - a\), but also, as we shall see in 
Section~\ref{OracleCaseSection} and Section~\ref{DataDriveEstimation}, 
is minimax optimal.
We also refer to \cite{moore1984} for a comprehensive literature review of set estimation in the one-dimensional case.

\subsubsection{Estimation of the volume of a convex set in high dimensions}
The one-dimensional model is useful to grasp the main ideas of volume estimation, yet it is not 
widely used in real applications. The two-dimensional model already covers several important applications in image analysis and signal processing. 
Here, we observe the points \(X_1,...,X_n\) drawn uniformly over a set \(C \subset \R^2\) and 
an objective is to recover the volume  \(V_C\) of the set and the set itself. Let us assume that 
\(C\) belongs to the class of convex sets. 
The wrapping hull is then simply the convex hull \(\Ch\) of the points \(\Ch = \text{conv}(X_1,...,X_n)\). 
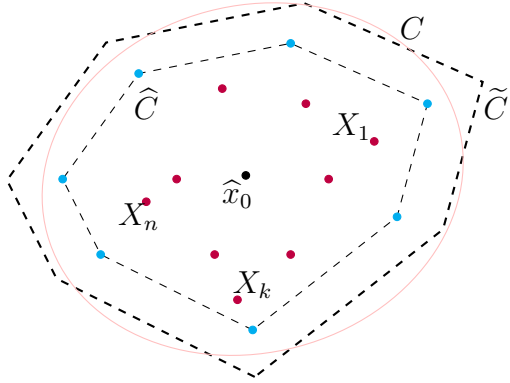
\begin{figure}[tp]
	\centering
	\begin{tikzpicture}
	\draw [dashed] (0,0) -- (-2,1) -- (-2.5,2) -- (-1.5,3.4) -- (0.5,3.8) -- (2.3, 3) -- (1.9, 1.5) -- (0,0);
	
	\draw [dashed,thick] (0.02734564, -0.6225691) -- (-2.58033532, 0.6812713) -- (-3.23225556,1.9851118) -- (-1.92841508,3.8104885) -- 
	(0.67926588, 4.3320247) -- (3.02617875,3.2889523) -- (2.50464256, 1.3331916) --  (0.02734564, -0.6225691);
	\draw [rotate around={15:(0,2)},pink] (0,2) ellipse (2.8 and 2.3);
	\draw [cyan, fill] (0,0) circle [radius=0.05];
	\draw [cyan, fill] (-2,1) circle [radius=0.05];
	\draw [cyan, fill] (-2.5,2) circle [radius=0.05];
	\draw [cyan, fill] (-1.5,3.4) circle [radius=0.05];
	\draw [cyan, fill] (0.5,3.8) circle [radius=0.05];
	\draw [cyan, fill] (2.3, 3) circle [radius=0.05];
	\draw [cyan, fill] (1.9, 1.5) circle [radius=0.05];
	\draw [black, fill] (-0.09, 2.049) circle [radius=0.05];
	\node at (-0.2, 1.8) {\(\hat x_0\)};
	
	\draw [purple, fill] (-1,2) circle [radius=0.05];
	\draw [purple, fill] (-0.5, 1) circle [radius=0.05];
	\draw [purple, fill] (-0.2, 0.4) circle [radius=0.05];
	\draw [purple, fill] (-1.4,1.7) circle [radius=0.05];
	\draw [purple, fill] (-0.4,3.2) circle [radius=0.05];
	\draw [purple, fill] (0.5,1) circle [radius=0.05];
	\draw [purple, fill] (0.7,3) circle [radius=0.05];
	\draw [purple, fill] (1,2) circle [radius=0.05];
	\draw [purple, fill] (1.6,2.5) circle [radius=0.05];

	\node at (2.1,4) {\(C\)};
	\node at (-1.4,3) {\(\hat{C}\)};
	\node at (3.2,3) {\(\tilde{C}\)};
	\node at (1.3,2.7) {\(X_1\)};
	\node at (0,0.6) {\(X_k\)};
	\node at (-1.5,1.5) {\(X_n\)};

	\end{tikzpicture}
	\caption{The points \(X_1,...,X_n\) drawn uniformly over a  set \(C\), the convex hull of the points \(\Ch = \text{conv}(X_1,...,X_n)\) and the dilated hull estimator 
		\(\tilde{C}\).}
	\label{HB}
\end{figure}
Analogously to the one-dimensional case, it is natural to consider the volume \(|\hat{C}|\) 
of the convex hull 
as a baseline estimator for the volume \(V_C\) of the set \(C\). It is quite intuitive that this estimator performs quite poorly
because it always underestimates the true volume and it should therefore be dilated as in the one-dimensional case. 
Section~\ref{OracleCaseSection} and Section~\ref{DataDriveEstimation} show that an optimal estimator has the following form
\begin{EQA}[c]
	\hat{V}_{opt} = \frac{n + 1}{n_\circ + 1} |\hat{C}|\,, 
	\label{hvopt}
\end{EQA}
where \(n_\circ\) is the number of purple points in  Figure~\ref{HB} that lie in the interior of the convex hull \(\hat{C}\). 
Note that \(\hat{V}_{opt}\) is the volume of the ``dilated'' hull \(\tilde{C}\), the set obtained by 
dilating the convex hull with the same factor from the centre of gravity \(\hat x_0\) of the convex hull:
\begin{EQA}[c]
	\tilde{C} 	= \Big\{\hat x_0+\Bigl(\frac{n+1}{n_\circ+1}\Bigr)^{1/d}  (x-\hat x_0)\,\Big|\,x\in\Ch\Big\} \,,
\end{EQA}
which can in fact  be used to estimate the set \(C\) itself.
Similarly, the same estimators for the volume and the set itself can be used in higher dimensions. 

The uniform model of a fixed number of  points drawn uniformly over a convex set \(C\) has been extensively studied in stochastic geometry.
The focus of study is rather on understanding the distributional characteristics of key functionals like the volume of \(\Ch\), the number of vertices of \(\Ch\) and 
the distance between \(\Ch\) and \(C\). The main references here are \cite{barany1988convex, reitzner2003random, reitzner2005central, vu2005sharp, pardon2011}.
The Poisson point process (PPP) model studied in the present paper 
is closely related to the uniform model. Using Poissonisation and de-Poissonisation techniques, this model exhibits asymptotic properties like the uniform model, see e.g. the references above and Section~\ref{UnifModelVolEst}. 
However the geometric properties of the PPP model are much more fecund for conducting statistical inference, see
\cite{ReissSelk14,BaldinReiss16},
where the techniques from the Poisson point processes theory were successfully employed for estimation of linear functionals 
in a one-sided regression model and estimation of the volume of a convex set.

\subsubsection{How fast can we estimate ${\pi}$?} 
There are quite a few ways how one can calculate the number \(\pi\), see \cite{arndt2001pi}. 
We here discuss one interesting way based on the Monte Carlo simulations of independent uniformly distributed random variables. 
It is a toy illustrative application of volume estimation in sampling theory.
Let us draw the points \(X_1,...,X_N\) from the uniform distribution over the square \([0,1] \times [0,1]\)
and count the number of points \(n\) which fall inside the circle centred at the origin of radius \(1\). 
Let \(\hat{\pi}:= n / N\) denote the ratio of the points inside the circle to the total number of points. 
It approximately equals \(\pi/4\), because it is an unbiased estimator:
\begin{EQA}[c]
	\E[\hat{\pi}] = \frac{1}{N} \E[n] = \frac{1}{N} \E\big[\sum_{i = 1}^{N} \Ind(X_i \in C)\big] = \frac{\pi}{4}\,,
\end{EQA}
and therefore its mean squared risk is governed  by the variance:
\begin{EQA}[c]
	\E\big[(\hat{\pi} - \pi)^2\big] = \Var(\hat{\pi}) = \frac{1}{N^2} \Var\big(\sum_{i = 1}^{N} \Ind(X_i \in C)\big) = \frac{1}{N}\frac{\pi}{4} \big(1 - \frac{\pi}{4}\big)\,.
\end{EQA}
It turns out \(\hat{\pi}\) is even a maximum likelihood estimator. Surprisingly,  we are able to 
estimate \(\pi\) with a much faster rate based on the data points in this experiment. Following \eqref{hvopt},
we define our properly scaled estimator as
\begin{EQA}[c]
	\hat{\pi}_{opt} = 4 \frac{n + 1}{n_\circ + 1} |\hat{C}|\,, 
\end{EQA}
where \(n_\circ\) is the number of points lying inside the convex hull \(\hat{C}\) of the points lying inside the circle.
Theorem~\ref{UpBoundOracle} and Theorem~\ref{new_estimator_risk_wrap} to follow state that the
rate of convergence of the mean squared risk of the estimator \( \hat{\pi}_{opt} \) satisfies
\(\E\big[(\hat{\pi}_{opt} - \pi)^2\big] = \bigO(N^{-5/3})\), see Figure~\ref{pi} for a numerical comparison of the two estimators. 
Note that both estimators can well be computed in polynomial time.
\begin{figure}[tp]
	\includegraphics[width=\linewidth]{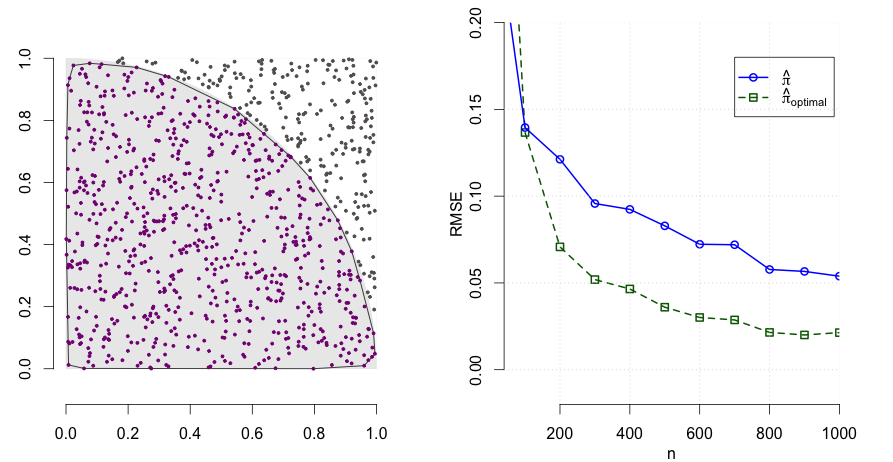}
	\caption[]{On the left: a sample of \(n = 500\) points drawn uniformly over  the square \([0,1] \times [0,1]\).
		On the right: Monte Carlo root mean squared error (RMSE) estimates for the studied estimators for \(\pi\) based on 200 Monte Carlo 
		simulations in each case.}
	\protect\label{pi}
\end{figure}

\subsection{Relationship to the work on volume estimation of a convex set} 

Some of the theoretical results in the present paper are underpinned by the results for the convex case, cf.
Section~\ref{OracleCaseSection} and Theorem~\ref{new_estimator_risk_wrap} in
Section~\ref{DataDriveEstimation}, and the corresponding results  in \cite{BaldinReiss16}. 
In fact, 
a key observation igniting the development of the present framework is that estimation of the volume of convex sets 
can in fact solely rely upon the property of convex sets being stable under taking intersections rather than convexity itself.
This striking observation appears to have a substantial value for volume estimation in a variety of scenarios far beyond convexity.
Volume estimation for some of the classes covered by the framework, in particular, the weakly-convex sets, has been long seen as notoriously
difficult with standard geometric arguments, see  the references in Section~\ref{ClassesOfBodies}. 

Not violating the flow of the paper, we shall therefore omit some of the proofs 
of the statements that are deduced from the proofs 
of the corresponding statements for the convex case.
The proof of the result that the wrapping hull is a complete statistic in Theorem~\ref{SuffComplThm} is slightly simplified compared to the proof of the theorem that the convex hull is a complete statistic in  \cite{BaldinReiss16} and hinges  upon 
a measure-theoretic result in stochastic geometry. 
In contrast to the special case of convex sets, the present paper further argues 
that the designed estimator \(\hat\vartheta \) is in fact \emph{adaptive} as its rate explicitly depends on
the rate of convergence of the missing volume \(|A \setminus \Ah|\).
This result rests upon Efron's inequality proved in Section~\ref{EfronWrapSect}.  
Section~\ref{UnifModelVolEst} explicitly states that the same estimator is minimax optimal in the uniform model. Section~\ref{ClassesOfBodies} appears to convey the most noticeable value for applications as it provides efficient data-driven estimators and clearly outlines the steps of deriving explicit rates of convergence for specific classes of intersection stable sets.


\section{Poisson point process theory and the wrapping hull} 
\label{PPPTheory}
In this section, we slightly digress on Poisson point processes, introduce the main notions and collect recently developed mathematical tools.
Let \((\Omega, \F, \P)\) be a probability space, fix a convex compact set \(\Eb \) in \(\R^d\) 
and equip it with the Borel \(\sigma\)-algebra \(\Ec\) with respect to the Euclidean metric \(\rho\). Without loss of generality one may assume that \(\Eb = [0,1]^d\). 
By a \emph{point process} \(\Nc\) on \(\Eb\) we mean an integer-valued random measure, or a kernel, 
from the probability space \((\Omega, \F, \P)\) into \((\Eb, \Ec)\). Thus, \(\Nc\)
is a mapping from \(\Omega \times \Ec\)  into \(\{0,1,...\}\) such that \(\Nc(\omega, \cdot)\) is 
an integer-valued measure for fixed \(\omega \in \Omega\) and \(\Nc(\cdot, B)\) is  an integer-valued random variable 
for fixed \(B \in \Ec\). For convenience, we write \(\Nc(B) = \Nc(\cdot, B)\). 

Let \(\Kb\) be the set of all compact subsets of \(\Eb\) equipped with its
Borel \(\sigma\)-algebra \(\Bf_{\Kb}\)
with respect to the Hausdorff-metric $\rho_H$ defined for two non-empty compact sets \(A\) and \(B\) by 
\begin{EQA}[c]
	\rho_H(A,B) = \max \big( \sup_{x \in B} \rho(x, A), \sup_{x \in A} \rho(x, B)\big)\,.
\end{EQA}
It is known, see Theorem C.5 in \cite{molchanov2006}, that the  Borel \(\sigma\)-algebra \(\Bf_{\Kb}\) coincides with 
the \(\sigma\)-algebra \(\sigma([B]_{\Kb}, B \in \Kb)\) with \([B]_{\Kb} = \{A \in \Kb : A \subseteq B \}\).  Moreover, the space \((\Kb, \rho_H) \) is Polish. 

\vspace{10pt}
Let \(\Ab \subset \Kb\) be a family of compact subsets of \(\Eb\) fulfilling  the following assumption	
\begin{assumption}
	\label{Assumption}
	\(\Ab\) is closed under taking arbitrary intersections and \(\varnothing, \Eb \in \Ab\).
\end{assumption}
\vspace{10pt}

Then the metric subspace \((\Ab, \rho_H)\) has 
the induced Borel \(\sigma\)-algebra \(\Bf_\Ab = \Ab \cap \Bf_\Kb = \{\Ab \cap K : K \in \Bf_\Kb \}\), which thus 
coincides with the \(\sigma\)-algebra
\(\Ac = \sigma([B], B \in \Ab)\) where \([B] = \{A \in \Ab: A \subseteq B \}\). 
It turns out there is a fascinating connection between the families of sets satisfying Assumption~\ref{Assumption} and the \emph{trapping systems} introduced in the groundbreaking  work of \cite{kendall1974foundations} on the theory of random sets, see also Section 7.2 in \cite{molchanov2006}.

We call a point process \(\Nc : \Omega \to M(\Eb)\) a Poisson point process (PPP) of intensity \(\lambda > 0 \) on \(A \in {\mathbf A}\), if 
\begin{itemize}
	\item for any \(B \in  \mathcal{E}\), we have \(\Nc(B) \sim \text{Poiss}\bigl(\lambda |A\cap B|\bigr)\), where $\abs{A\cap B}$ denotes the Lebesgue measure of $A\cap B$;
	\item for all mutually disjoint sets \(B_1,..., B_n \in    \mathcal{E}\), the
	random variables \newline \(\Nc(B_1),...,\Nc(B_n) \) are independent.
\end{itemize}
For statistical inference, we assume the Poisson point process to be defined on a set of non zero  Lebesgue measure, i.e. \(|A| >\delta\) for 
\(\delta >0\).

Interestingly, one can view \( (\Nc(K),K \in \Kb)\) as a set-indexed stochastic process. It has no direct order 
and its natural filtration is defined by 
\begin{EQA}[c]
	\F_K \eqdef \sigma(\{\Nc(U); U \subseteq K,  U \in \Kb \})
\end{EQA}
for any \( K \in \Kb\). The properties of  the filtration \((\F_K, K \in \Kb)\) are well studied, cf 
\cite{zuyev1999stopping}. By construction, the restriction \(\Nc_K = \Nc(\cdotp \cap K)\) of the point process \(\Nc\) onto
\(K \in \mathbf{K}\) is \(\F_K\)-measurable (in fact, \(\F_K = \sigma(\{\Nc_K(U); U \in \mathbf{K}\})\)). Moreover, it can be easily seen that \(\Nc_K\) is a Poisson point process in \(\mathbf{M}\), cf. the
Restriction Theorem in \cite{kingman1992poisson}. 
Furthermore, we say that a set-indexed, \((\F_K)\)-adapted integrable process
\((X_K, K \in \Kb)\) is a martingale if  \(\E[X_B| \F_A] = X_A\) holds for any \(A, B \in \Kb\) with \(A \subseteq B\).
By the independence of increments, the process 
\begin{EQA}[c]
	M_K \eqdef \Nc(K) - \lambda|K|,\quad K\in\Kb\,,
\end{EQA}
is clearly a martingale with respect to its natural filtration \((\F_K,K\in\Kb)\).

A random compact set \(\Kh\) is a measurable mapping \(\Kh:  (\mathbf{M},\mathcal{M}) \to (\Kb,\Bf_\Kb)\).
A random compact set \(\Kh\) is called an \(\F_K\)-stopping set if \(\{\Kh \subseteq K\} \in \F_K\)
for all \(K \in \Kb\). The sigma-algebra of \(\Kh\)-history is defined as
$\F_{\Kh} = \{A \in \mathfrak{F}: A \cap \{ \Kh \subseteq K \} \in \F_K\,\,\, \forall K \in \Kb \}$, where
\(\mathfrak{F} =   \sigma(\F_K; K \in \Kb)\).
We introduce the \emph{wrapping hull} of the PPP points on a set \(A \in \Ab\), which is served as a set-estimator of \(A\).

\vspace{10pt}
\begin{definition}
	\label{WrapDef}
	The \(\Ab\)-\emph{wrapping hull} (or simply the wrapping hull) of the PPP points is a mapping \(\Ah: M \to \Ab\) defined as
	\begin{EQA}[c]
		\Ah \eqdef \wrap\{X_1,...,X_N\} \eqdef \bigcap \{A \in \Ab:  X_i \in A, \forall i= 1,...,N \}\,.
	\end{EQA}
\end{definition}
\vspace{10pt}

For a set $A\subset {\mathbf E}$ let $A^c$ denote its complement.

\begin{lemma}
	\label{LemmaStoppingSet}
	The set $\hat\Kh\eqdef\overline{\Ah^c}$, the closure of the complement of the wrapping hull, is an $(\F_K)$-stopping set.
\end{lemma}
The proof of Lemma~\ref{LemmaStoppingSet} essentially repeats the steps of the proof of an analogous statement for the convex hull of the PPP points given in Lemma 2.2 in \cite{BaldinReiss16}. A striking observation is the following corollary, which rests upon 
the optional sampling theorem for set-indexed martingales, cf. \cite{zuyev1999stopping}. 
\begin{corollary}
	\label{EfronsIdenPoisson}
	The number of points \(N_\partial\) lying on the wrapping hull \(\Ah\) and
	the missing volume \(|A \setminus \Ah|\) satisfy the relation
	\begin{EQA}[c]
		\label{efronsIdent2}
		\E[N_\partial] = \lambda\, \E[ |A\setminus \Ah|]\,.
	\end{EQA}
\end{corollary}

We shall define the \emph{likelihood function} for the PPP model. Note that we can evaluate the probability
\[
\P_A\bigl(\Ah \in  B\bigr)
= \sum_{k = 0}^{\infty} \frac{\e^{-\lambda \abs{A}} \lambda^k}{k!}
\int_{A^k} {\bf 1}(\wrap\{x_1, ..., x_k\} \in B ) \d (x_1,..., x_k)
\]
for \(B \in \Bf_{\Ab}\). Usually, we only write the subscript $A$ or sometimes $(A,\lambda)$ when different probability distributions are considered simultaneously.
The likelihood function \(\Lc(A, \Nc) = \frac{\d \P_{A,\lambda}}{\d \P_{\mathbf{E},\lambda_0}}\) for \(A \in \Ab\) and $\lambda,\lambda_0>0$
is then given by
\begin{EQA}[c]
	\Lc(A, \Nc) = \frac{\d \P_{A,\lambda}}{\d \P_{\mathbf{E},\lambda_0}}(\Nc) 
	= \ex^{\lambda_0|{\mathbf E}|-\lambda|A|}(\lambda/\lambda_0)^N {\bf 1}(\Ah \subseteq A)\,,
	\label{fPPeb}
\end{EQA}
cf. Thm. 1.3 in \cite{Kutoyants1998}. For the last line, we have used that a point set is in $A$ if and only if its wrapping hull 
\(\Ah = \wrap\{X_1,...,X_N\}\) is contained in $A$.

The following theorem is an essential ingredient for deriving our estimators. 
The statement
of the theorem is quite intuitive and already used in \cite{privault2012invariance}. Its proof is similar to the proof of Theorem 3.1 in \cite{BaldinReiss16}.
\begin{theorem}
	\label{LemMeasCond}
	Set $\hat\Kh\eqdef\overline{\Ah^c}$.
	The number \(N_{\partial}\) of points on the boundary of the wrapping hull \(\Ah\) is measurable with respect to the sigma-algebra of
	\(\hat\Kh\)-history \(\F_{\hat\Kh}\). The number of points in the interior of the wrapping hull \(N_\circ\) is, conditionally on \(  \F_{\hat\Kh}\), Poisson-distributed:
	\begin{equation}
	\label{NccCs}
	N_\circ \cond \F_{\hat\Kh} \sim \text{Poiss}(\lambda_\circ)\text{ with }\lambda_\circ \eqdef \lambda | \Ah |.
	\end{equation}
	In addition, we have \(\F_{\hat\Kh} = \sigma(\Ah)\), where
	the latter is the sigma-algebra \( \sigma(\{\Ah \subseteq B, B \in \Ac\})\) completed with the null sets in  \(\mathfrak{F}\).
\end{theorem}

We shall further use the following short notation:  \(N =\Nc(\Eb)\) denotes the total number of points, \(N_\circ = \Nc({\Ah}^\circ)\) the number of points in the interior
of the wrapping hull \(\Ah\) and \( N_{\partial} = \Nc(\partial\Ah)=\Nc(\partial\hat\Kh) \) the number of points on the boundary of the wrapping hull.
For asymptotic bounds we write \(f(x) = O(g(x))\) or \(f(x) \lesssim g(x)\) if \(f(x)\) is bounded by a constant multiple of \(g(x)\)  and
\(f(x) \thicksim g(x)\) if \(f(x) \lesssim g(x)\)  as well as \(g(x) \lesssim f(x)\).

\section{Oracle case: known intensity $\lambda$}\label{OracleCaseSection}
In the case when \(\lambda\) is known one can just estimate the volume \(|A|\) by \(N/\lambda\), which is an unbiased estimator,
whose mean squared risk is given by 
\begin{EQA}[c]
	\label{RiskNaiveEst}
	\E[(N / \lambda - |A|)^2] = \Var(N/\lambda) = |A| / \lambda \,,
\end{EQA}
thus implying \(\bigO (\lambda^{-1/2})\)-rate of convergence. This rate can be improved. 
As we shall see in Theorem~\ref{SuffComplThm}, the wrapping hull is a complete and sufficient statistic thus allowing 
to construct the unique best unbiased estimator of the volume in virtue of the Lehmann-Scheff\'e theorem.
In view of Theorem~\ref{LemMeasCond} we thus derive our oracle estimator 
\begin{EQA}[c]
	\hvo=
	\E \Big[\frac{N}{\lambda} | \F_{\hat\Kh} \Big]  =  \E \Big[\frac{N_\circ + N_\partial}{\lambda} | \F_{\hat\Kh} \Big] = \frac{N_\partial}{\lambda} + |\Ah|
	\label{OracleEst}
\end{EQA}
The following result is fundamental in characterising the rate of convergence of the risk of the oracle estimator. 
\begin{theorem}
	\label{UpBoundOracle} For known intensity \(\lambda > 0\), the oracle estimator \(\hvo\) is unbiased and of minimal variance among all unbiased estimators (UMVU)
	in the PPP model with parameter class \(\Ab\). It satisfies
	\begin{EQA}[c]
		\label{ORACLEEst}
		\E \big[(\hvo- |A|)^2\big] = \Var(\hvo) = \frac{\E[|A\setminus \Ah|]}{\lambda}\,.
	\end{EQA}
\end{theorem}
\begin{remark} The theorem  asserts that the rate of convergence of \(\hvo\) is in fact faster than \(\lambda^{-1/2}\) for all classes of sets \(\Ab\) satisfying  
	Assumption~\eqref{Assumption}.
\end{remark}

\begin{proof}
	By the tower property of conditional expectation, the estimator \(\hvo\) is unbiased, \(\E[\hvo] = |A|\). Using law of total variance, we derive 
	\begin{EQA}
		\Var(\hvo) = \Var\Big(\E \Big[\frac{N}{\lambda} | \F_{\hat\Kh} \Big]\Big)& =& \Var(\frac{N}{\lambda}) - \E \Big[\Var\Big(\frac{N}{\lambda}| \F_{\hat\Kh} \Big)\Big] \\
		&=& \frac{|A|}{\lambda} - \E \Big[\Var\Big(\frac{N_\circ + N_\partial}{\lambda}|\F_{\hat\Kh} \Big)\Big] = \frac{\E[|A\setminus \Ah|]}{\lambda}\,.
	\end{EQA}
	Theorem~\ref{SuffComplThm} below affirms that the wrapping hull \(\Ah\) is a 
	complete and sufficient statistic such that by the Lehmann-Scheff\'e theorem, the
	estimator \(\hvo\) has the UMVU property.
\end{proof}

\begin{theorem}
	\label{SuffComplThm} For known intensity \(\lambda > 0\), the wrapping hull  is a complete and sufficient statistic.
\end{theorem}
The proof of Theorem~\ref{SuffComplThm} is deferred to the Appendix. 
As a result of Theorem~\ref{UpBoundOracle}, the performance of the estimator \(\hvo\) of the volume is reduced to the analysis of the performance of the wrapping hull 
estimator of the set itself, which clearly depends on the geometric properties of classes of sets satisfying 
Assumption~\ref{Assumption}.

The minimax lower bounds on the rate of convergence of the risk of estimating the volume of a  set \(A \in \Ab\)  are often easier to establish for concrete classes of sets using the so-called hypercube argument, cf.  \cite{gayraud1997}. Interestingly, the following general bound on the minimax optimal rate holds.
\begin{theorem}
	\label{MinimaxLowerBoundThm}
	The minimax optimal rate of estimating the volume of a set \(A \in \Ab\)  satisfies 
	\begin{EQA}[c]
		\label{eq:ddvvvbe}
		\lambda^{-2} \lesssim	\inf_{ \hat{\vartheta}_\lambda} \sup_{A \in \Ab}	\E \big[(\hat{\vartheta}_\lambda - |A|)^2\big] \lesssim \lambda^{-1}\,,
	\end{EQA}
	where the infimum  extends over all estimators \(\hat{\vartheta}_\lambda \) in the Poisson point process model with intensity \(\lambda\).
\end{theorem}
\begin{remark} The rate \(	\bigO(\lambda^{-1})\) is minimax for estimating the volume in some parametric classes of sets, in particular, the class of concentric sets, whereas the rate \(\bigO( \lambda^{-1/2})\) is established for estimating the volume in the class of compact sets, see Section~\ref{ClassesOfBodies}.
\end{remark}
\begin{proof}
	The upper bound in \eqref{eq:ddvvvbe} follows directly from Theorem~\ref{UpBoundOracle}.
	The lower bound is obtained by reducing the minimax risk to the Bayes risk and then lower-bounding the Bayes risk at its minimum. These steps are fairly standard, 
	cf. \cite{korostelev1993minimax}, and we hence omit them here.
\end{proof}

\section{Data-driven estimator of the volume}
\label{DataDriveEstimation}
The main ingredient to deriving the estimator of \(\lambda\) is the fact that 
the closure of the complement of the \(\Ab\)-wrapping hull \(\hat\Kh\eqdef\overline{\Ah^c}\) 
is in fact an $(\F_K)$-stopping set according to Lemma~\ref{LemmaStoppingSet}. 
Moreover in analogy with a time-indexed Poisson process, 
our problem boils down to the estimation of the intensity of a time-indexed Poisson process 
starting from an unknown origin. To see this, recall that  according to Theorem~\ref{LemMeasCond},
the number of points \(N_\circ\) lying inside the wrapping hull \(\Ah\) is Poisson-distributed with intensity \(\lambda_\circ \eqdef \lambda | \Ah |\) 
provided that \(|\Ah| > 0\):
\begin{EQA}[c]
	N_\circ \cond \F_{\hat\Kh} \sim \text{Poiss}(\lambda_\circ).
\end{EQA}
We aim to find an estimator for \(\lambda_\circ^{-1}\). On the event \(|\Ah| > 0\), we follow the idea developed in \cite{BaldinReiss16}. 
That is to say, we use that the first jump time \(\tau\) of a time-indexed Poisson process \((Y_t, t > 0)\) with intensity \(\nu > 0\) 
is \(\text{Exp}(\nu)\)-distributed and hence \(\E[\tau] = \nu^{-1}\). Using the memoryless property of the 
exponential distribution, we then have 
\begin{EQA}[c]
	\E[\tau | Y_1 = m] = \frac{1}{m+1}\Ind(m \ge 1)+ (1 + \nu^{-1})\Ind(m = 0)\,.
\end{EQA}
Therefore, we conclude that 
\[ \hat\mu(N_\circ, \lambda_\circ) \eqdef \begin{cases} (N_\circ+1)^{-1},& \text{ for }N_\circ\ge 1,\\ 1+\lambda_\circ^{-1},&\text{ for } N_\circ=0
\end{cases}
\]
satisfies \(\E[\hat\mu(N_\circ, \lambda_\circ) | \F_{\hat\Kh}] = \lambda_\circ^{-1}\) provided that \(|\Ah| > 0\). Omitting the term 
depending on \(\lambda_\circ\) in the unlikely event \(N_\circ = 0\), we derive our final estimator:
\begin{EQA}[c]
	\label{FinEst}
	\hat\vartheta = |\Ah| + \frac{N_\partial}{N_\circ + 1} |\Ah| = \frac{N + 1}{N_\circ + 1} |\Ah|\,. 
\end{EQA}
\begin{remark}
	As it follows from Definition~\ref{WrapDef}, a wrapping hull \(\Ah\) may consist of disjoint sets, in which case 
	the number of points \(N_{\circ,k}\) lying inside a piece \(\Ah_k\) satisfies 
	\(N_{\circ,k} \cond \F_{\hat\Kh} \sim \text{Poiss}(\lambda |\Ah_k|)\) due to the homogeneity of the Poisson point process. 
	This fact can further be used to estimate \(\lambda^{-1}\) locally. However, in the homogeneous case, 
	we prefer to use the total number of points to estimate the intensity.
\end{remark}

Note that a more explicit bound can be derived using the Cauchy-Schwarz inequality given  a bound on the 
expected number of points \(N_{\partial}\) lying on the wrapping hull \(\Ah\) and  a bound on the moments of the missing volume \( | A \setminus \Ah|\).
This clearly depends on a considered class of sets satisfying Assumption~\ref{Assumption}. The following very general oracle inequality holds 
for the mean squared error of the estimator \( \hat\vartheta\). Its proof 
 can be adapted from the proofs of Thm.~4.4 and Thm.~4.5  in \cite{BaldinReiss16} and we hence omit it here.

\begin{theorem}
	\label{new_estimator_risk_wrap}
	The following oracle inequality for the risk of the estimator \( \hat\vartheta\) holds
	for all $A \in \Ab$ whenever $\lambda\abs{A}\ge 1$:
	\begin{EQA}[c]
		\E[(\hat\vartheta - |A|)^2]^{1/2} \le  (1+c \alpha(\lambda, A))\Var(\hat\vartheta_{oracle})^{1/2} +r(\lambda,A)\,,
	\end{EQA}
	where
	\begin{align*}
	\alpha(\lambda, A) &:= \frac{1}{\abs{A}} \Bigl(\frac{ \Var(\abs{A\setminus\Ah} )}{ \E[ \abs{A\setminus\Ah}]}
	+   \E[ \abs{A\setminus\Ah}] \Bigr)\,,\\
	r(\lambda,A) &:=c_1 \lambda^{-1} \E[N_{\partial}^4]^{1/4} \P(|A\setminus \Ah| \ge |A|/2)^{1/4}\,,
	\end{align*}
	with some numeric constants \(c,c_1 > 0\). In particular, \(\alpha(\lambda, A)\) is bounded by some universal constant.
\end{theorem}

\subsection{Volume estimation in the uniform model} 
\label{UnifModelVolEst}
In the PPP model, the data we observe are uniformly distributed points over a set in some given class and the number of points is a realisation of a Poisson random variable. The uniform model, 
\begin{EQA}[c]
	X_1, ..., X_n \overset{i.i.d.}{\sim} U(A), \, \quad A \in \Ab\,,
\end{EQA}
is closely related to the PPP model and assumes that the number of points \(n\) is fixed. In stochastic geometry, the objects studied in the PPP model typically exhibit a similar asymptotic behaviour in the uniform model and vice versa, see e.g. \cite{pardon2011} and references therein for a study of the functionals of the convex hull. This section
examines which results of the present paper derived in the PPP model remain true in the uniform model. 

It is relatively straightforward to show that the wrapping hull remains a sufficient and complete statistic in the uniform model with slightly adjusted arguments of the proof of Theorem~\ref{SuffComplThm}. It is unknown however 
whether there exists an UMVU estimator in the uniform model. Nevertheless, an estimator 
\begin{EQA}[c]
	\label{FinEstUniform}
	\hat\vartheta_{\text{unif},n} := \frac{n + 1}{n_\circ + 1} |\Ah|\,,
\end{EQA}
where \(n_\circ \) is the number of points lying inside the wrapping hull, inherits the 
same rate of convergence as the final estimator \(\hat\vartheta\) in \eqref{FinEst} in the PPP model due to the following result
\begin{proposition}[Poissonisation]
	Let \(n = \lfloor \lambda |A| \rfloor > 0\) with \(A\) being any set in a class \(\Ab\).
	Then letting \(\lambda \to \infty\) the following asymptotic equivalence result holds true
	\begin{EQA}[c]
		\label{PoissDePoissMinMax}
		\E \big[(\hat\vartheta_{\text{unif},n} - |A|)^2\big] \thicksim 
		\E \big[(\hat\vartheta - |A|)^2\big] \,, \quad \forall A \in \Ab\,.
	\end{EQA}
	Furthermore, the minimax lower bounds satisfy 
	\begin{EQA}[c]
		\label{PoissDePoissMinMax2}
		\inf_{ \hat{\vartheta}_n} \sup_{A \in \Ab}	\E \big[( \hat{\vartheta}_n - |A|)^2\big] \thicksim 
		\inf_{ \hat{\vartheta}_\lambda} \sup_{A \in \Ab}	\E \big[(\hat{\vartheta}_\lambda - |A|)^2\big] \,,
	\end{EQA}
	where the infimum on the left-hand side extends over all estimators in the uniform model, whereas the infinum on the right-hand side extends over all estimators in the Poisson point process model.
\end{proposition}

\begin{proof}
	We only prove \eqref{PoissDePoissMinMax2} here. \eqref{PoissDePoissMinMax} can then be proved exploiting similar arguments. Let us first show the inequality \( "\gtrsim" \). Assume it does not hold and that there exists an estimator \(\hat{\vartheta}_n^\prime \) in the uniform model with the rate of convergence faster than the minimax optimal rate in the PPP model. Then for the estimator 
	\(\hat{\vartheta}_{N}^\prime \) we have for any \(A \in \Ab\)
	\begin{EQA}
		\E \big[(\hat{\vartheta}_{N}^\prime  - |A|)^2\big] & =& \sum_{k = 1}^{\infty } \E \big((\hat{\vartheta}_{N}^\prime  - |A|)^2 \cond N = k\big)\, \P(N = k) \\ \\
		&\le& \sum_{k =  \lfloor \lambda |A|/2 \rfloor }^{ \lfloor 2 \lambda |A| \rfloor  } \E \big((\hat{\vartheta}_{N}^\prime  - |A|)^2 \cond N = k\big)\, \P(N = k)  + c_2 \exp(-c_3 n )\\
		&\le& c_1 \E \big[( \hat{\vartheta}_n^\prime - |A|)^2\big]  + c_2 \exp(-c_3 n)\,,
	\end{EQA}
	for some constants \(c_1,c_2,c_3 > 0\) using Bennett's inequality, a contradiction in view of Theorem~\ref{MinimaxLowerBoundThm}. The other direction follows using the same technique.
\end{proof}

\subsection{Efron's inequality  for the wrapping hull}
\label{EfronWrapSect}
In this section, we show that  the rate of convergence of the risk for the estimator \(\hat\vartheta\) in Theorem~\ref{new_estimator_risk_wrap} 
hinges in fact upon only a deviation of the missing volume  \(| A \setminus \Ah|\). 
More than 50 years ago Efron showed  in \cite{efron1965convex} that the moments of the number of the points \(N_{{\partial},k}\) lying on the 
boundary of a convex hull \(\Ch_{k}\) in the uniform model \(X_1,...,X_k \overset{i.i.d.}{\sim} U(C)\), with \(C \subset \R^d\)
being a convex set, satisfies the identity 
\begin{EQA}[c]
	\E[N_{{\partial},k}^q] = \sum_{r = 1}^{q} n(k, q, r) \E[|\Ch_{k-r}|^r]\,,
\end{EQA}
where \(n(k, q, r)\) is the number of \(q\)-tuples from \(1,...,k\) having exactly \(r\) different values, 
\(n(k, q, r) = {k \choose r} \sum_{m = 1}^{r} (-1)^{r - m}  {r \choose m} m^q\). This yields a striking dimension-free 
asymptotic equivalence result,
\begin{EQA}[c]
	\E[N_{{\partial},k}^q] \thicksim k^q \E[|C \setminus \Ch_{k}|^q]\,.
	\label{EfronsIneqRev}
\end{EQA}

We here extend a one-sided version of this results to the wrapping hull. 
\begin{proposition}[Efron's inequality  for the wrapping hull]
	\label{EfronsProp}
	Let \(\Ab\) be any class satisfying Assumption~\ref{Assumption} and \(\Ah\) be the corresponding wrapping hull of the PPP points 
	of intensity \(\lambda > 0\)
	over a set \(A \in \Ab\). 
	Then the following asymptotic inequality holds
	\begin{EQA}[c]
		\E[N_{\partial}^q] \lesssim \lambda^q \E[|A \setminus \Ah|^q]\,,
	\end{EQA}
	provided that the probability of observing \(q\) points lying on the boundary of the wrapping hull \(\Ah\) is non-zero.
\end{proposition}
\begin{remark} 
	It follows by Jensen's inequality and Corollary~\ref{EfronsIdenPoisson}, that 
	\(\E[N_{\partial}^q]  \ge \E[N_{\partial}]^q = \lambda^q \E[|A \setminus \Ah|]^q\). For some examples, 
	like the class of convex sets, this in fact implies \(\E[N_{\partial}^q] \thicksim \lambda^q \E[|A \setminus \Ah|^q]\).
\end{remark}
\begin{remark} 
	Identities that relate the functionals of the convex hull of the points distributed uniformly over a convex set are thoroughly studied in stochastic geometry, see \cite{pardon2011, buchta2013,reitzner2015}. 
	
\end{remark}

\begin{proof}
	Let us first consider the uniform model and then transfer the result to the PPP model using Poissonisation. 
	We follow Efron's idea, see also \cite{reitzner2015, brunel:tel-01066977}, that 
	\begin{EQA}
		\E[|A \setminus \Ah_k|^q] &=& |A|^q \, \P(X_{k + 1} \notin \Ah_{k},..., X_{k + q} \notin \Ah_{k} )  \\
		& \ge &  |A|^q\, \P(X_{k + 1} \in \partial \Ah_{k+q},..., X_{k + q} \in \partial \Ah_{k+q}) \\
		& = &  \frac{|A|^q}{{k + q \choose q} }\,  \E \sum_{}^{} \Ind(X_{i_1} \in \partial \Ah_{k+q},..., X_{i_q} \in \partial \Ah_{k+q}) \\
		&=& \frac{|A|^q}{{k + q \choose q} }\,  \E {N_{{\partial}, k+q} \choose q}\,,
	\end{EQA}
	the sum being taken over all tuples \((i_1, ..., i_q)\) from the integers \(1,..., k+q\). Rearranging the terms, this entails 
	\( \E [N_{\partial, k}^q] \lesssim k^q \E[|A \setminus \Ah_k|^q] \).
	
	Using Poissonisation, we further derive for the PPP model, 
	\begin{EQA}[c]
		\E[|A \setminus \Ah|^q] = \sum_{k = 1}^{\infty } \E(|A \setminus \Ah_N|^q \cond N = k)\, \P(N = k) \\
		\gtrsim \sum_{k = 1}^{\infty } (2 \lambda|A|)^{-q} \E\big[ N_{\partial, k}^q\big] \, \P(N = k) + 
		\sum_{k = \lfloor 2 \lambda|A| \rfloor}^{\infty}( k^{-q} - (2 \lambda|A|)^{-q} )  \E\big[ N_{\partial, k}^q\big] \, \P(N = k)\\ 
		=(2 \lambda|A|)^{-q}  \E[N_{\partial}^q]  + 
		\sum_{k = \lfloor 2 \lambda|A| \rfloor}^{\infty}( k^{-q} - (2 \lambda|A|)^{-q} )  \E\big[ N_{\partial, k}^q\big] \, \P(N = k)
	\end{EQA}
	with the absolute value of second sum being bounded using  the Cauchy-Schwarz inequality 
	and large deviations by 
	\begin{EQA}[c]
		c_1 \, \E[N_{\partial}^{2q}]^{1/2} \P (N \ge [2\lambda |A|])^{1/2} \le 
		c_1 \, \E[N_{\partial}^{2q}]^{1/2} \exp(-c_2 n )\,,
	\end{EQA}
	for some constants \(c_1, c_2 >0 \). Thus, \eqref{EfronsIneqRev} follows.
\end{proof}
Proposition~\ref{EfronsProp} and Theorem~\ref{new_estimator_risk_wrap} immediately suggest the following bound for the remainder term in
the oracle inequality,
\begin{EQA}[c]
	r(\lambda,A) \le c   \E[|A \setminus \Ah|^4]^{1/4}\, \P(|A\setminus \Ah| \ge |A|/2)^{1/4}\,,
\end{EQA} 
for some numeric constant \(c>0\). Therefore, the oracle inequality in Theorem~\ref{new_estimator_risk_wrap} 
hinges upon only two probabilistic results:
\begin{itemize}
	\item the ratio \(\Var(\abs{A\setminus\Ah})  / \E[ \abs{A\setminus\Ah}]\) of the moments of the missing volume,
	\item a uniform deviation inequality for the missing volume.
\end{itemize}
Both results are fairly involved and we shall only discuss here how to derive them for some simple classes of sets satisfying 
Assumption~\ref{Assumption}.

\section{Classes of sets satisfying Assumption~\ref{Assumption}}
\label{ClassesOfBodies}
This section collects some examples of classes of sets that satisfy  Assumption~\ref{Assumption}. Note that 
the class of all convex sets \(\Cb_{\text{conv}}\) satisfies the assumption and was extensively studied in \cite{BaldinReiss16}. 
The most involved statements 
in the inference on convex sets were underpinned by the abundance of results from stochastic geometry on moment bounds 
and deviation inequalities for the missing volume, see Lemma 4.6 in \cite{BaldinReiss16}. In particular,
the ratio \(\Var(\abs{C\setminus\Ch})  / \E[ \abs{C\setminus\Ch}] \thicksim 1/\lambda\) is established in 
\cite{pardon2011} for all convex sets \(C\)  in dimensions \(d = 1,2\). In dimensions \(d > 2\),
one can bound the ratio only for some subsets of the class of convex sets. 
Thus, for a convex set \(C\)  with $C^2$-boundary of positive curvature, it is known thanks to \cite{reitzner2005central} that   \(\Var(\abs{C\setminus\Ch} ) \lesssim \lambda^{-(d+3)/(d+1)}\). The lower bound
for the first moment, \( \E[ \abs{C\setminus\Ch}] \gtrsim \lambda^{-2/(d+1)}\), was shown in \cite{schutt1994random}.
For  a polytope \(C\),   the upper bound  \(\Var(\abs{C\setminus\Ch} ) \lesssim \lambda^{-2}(\log \lambda)^{d-1}\) was obtained in \cite{barany2010variance},
while the lower bound for the first moment, \( \E[ \abs{C\setminus\Ch}] \gtrsim \lambda^{-1}(\log \lambda)^{d-1}\), was proved in \cite{barany1988convex}.
A uniform deviation inequality for convex sets
obtained in Thm.~1 in  \cite{Bru14c} allows to derive sharp upper bounds on the moments of the missing volume.
The proof of the deviation inequality exploited a bound on the entropy of convex sets.
It remains an intriguing open question in stochastic geometry whether \( \lambda \Var(\abs{C\setminus\Ch}) \thicksim \E[ \abs{C\setminus\Ch}]  \)
holds universally for all convex sets in arbitrary dimensions.
Some of the classes of sets we consider here are much larger, yet very little has been known about them in the mathematical literature. 
\subsection{$r$-convex sets}
\label{RConnSect}
We denote by \(B(x,r) \subset \R^d\)  (resp. \(B_\circ(x,r)\)) the closed (resp. open) ball with centre \(x\)  and radius \(r\).
\begin{definition}
	A compact set \(C_r\) in \(\Eb \subset \R^d\) is called \emph{\(r\)-convex} for \(r > 0\), if its complement is the union 
	of all open Euclidean balls of diameter \(r\) that are disjoint to \(C_r\), i.e. if
	\begin{EQA}[c]
		C_r = \bigcap_{B_\circ^c(x,r) \cap C_r = \varnothing} B_\circ^c(x,r)\,.
	\end{EQA}
	We denote the class of \(r\)-convex sets by \(\Cb_r\).
\end{definition}
Note that an \(r\)-convex set fulfills the \emph{outside rolling ball condition}, i.e. for all \(y \in \partial C_r\) there is a closed ball \(B(x,r)\) 
such that \(y \in \partial B(x,r)\) and \(B_\circ(x,r) \cap C_r = \varnothing\). Heuristically this means 
that one can ``roll'' a ball of radius \(r\) freely over the boundary of a set. Note that according to the definition, \(r\)-convex 
sets can have ``holes'' and do not need to be connected, see Figure~\ref{r-convex_alpha_varies} for some examples. 
In the terminology of \cite{kendall1974foundations}, \(C_r \in \Cb_r \) means that the set \(C_r\) is \emph{trapped} by the balls of 
radius \(r\).
The \(r\)-convex sets were introduced in
\cite{Perkal1956} and presumably independently in \cite{efimov1959some}; see \cite{walther1997,cuevas2012} and references therein for a recent work on estimation of \(r\)-convex sets.
In the literature, much more attention has been devoted to the sets satisfying the so-called \emph{inside and outside rolling ball condition}, when both 
\(C_r\) and \(\overline{C_r^c}\) are \(r\)-convex, see \cite{mammen1995asymptotical,walther1996}. The reason probably is that sets with 
smooth boundaries (with no angles) are sometimes easier to handle with geometric arguments, see \cite{lopez2008}. 

The \(\Cb_r\)-wrapping hull is defined by 
\begin{EQA}[c]
	\Ch_r := \bigcap_{B_\circ^c(x,r) \cap \{X_1,...,X_N\} = \varnothing} B_\circ^c(x,r)
\end{EQA}
and often called the \emph{\(r\)-convex hull} in the literature. 
Thus the oracle estimator in \eqref{OracleEst} has the following form 
\begin{EQA}[c]
	\label{r_convex_hull_estimator}
	\hv_{r,oracle} := \frac{N_\partial}{\lambda} + |\Ch_r|\,,
\end{EQA}
where \(N_\partial\) is the number of sample points lying on the the \(r\)-convex hull \(\Ch_r\). In order to investigate the performance 
of this estimator according to Theorem~\ref{UpBoundOracle} it suffices to study \( \sup_{C_r \in \Cb_r}\E_{C_r}[|C_r\setminus \Ch_r|]\). 
In fact the following result holds and it is a consequence of Theorem~\ref{UpBoundOracle}.
\begin{theorem}
	\label{ThmRconv}For known intensity $\lambda > 0$, 
	the worst case mean squared error of the oracle estimator \(\hv_{r,oracle}\) over the parameter class $\Cb_r$ decays as $\lambda\uparrow \infty$  like $\sup_{C_r \in \Cb_r}\E_{C_r}[|C_r\setminus \Ch_r|] / \lambda$ in dimension $d$:
	\begin{EQA}[c]
		\limsup_{\lambda\to \infty} \lambda \sup_{C_r \in \Cb_r, |C_r| >0} \Big\{\E\big[(\hv_{r,oracle} -\abs{C_r})^2\big] /\E\big[|C_r\setminus \Ch_r|\big]\Big\}  < \infty\,.
	\end{EQA}
\end{theorem}
\begin{remark}
	Note that the class of convex sets \(\Cb_{\text{conv}}\) belongs to \(\Cb_r\) for all \(r>0\)
	and thus using Theorem 3.4 in \cite{BaldinReiss16} we have a lower bound on the rate of convergence,
	\begin{EQA}
		&&\inf_{ \hv_\lambda}  \lambda^{(d+3)/(d+1)} \sup_{C_r \in \Cb_r}\E_{C_r } [ ( | C_r | -\hv_\lambda)^2] \\
		&&\quad \quad   \ge 
		\inf_{ \hv_\lambda}  \lambda^{(d+3)/(d+1)} \sup_{C \in \Cb_{\text{conv}}}\E_{C } [ ( | C | -\hv_\lambda)^2] > 0\,,
	\end{EQA}
	where the infimum extends over all estimators \(\hv_\lambda\) in the PPP model with intensity \(\lambda\).
	Furthermore,  the rate $\lambda^{-(d+3)/(d+1)}$ is achieved up to a logarithmic factor for sets \(C_r \in \Cb_r \) 
	with a smooth boundary following \cite{lopez2008}.
\end{remark}
Following Section~\ref{DataDriveEstimation}, the mean squared error of the estimator 
\begin{EQA}[c]
	\label{FinEstRConv}
	\hat\vartheta_r := \frac{N + 1}{N_\circ + 1} |\Ch_r|\,,
\end{EQA}
satisfies the oracle inequality in Theorem~\ref{new_estimator_risk_wrap}. Upper bounding  
the functions \(\alpha(\lambda, C_r)\) and \(r(\lambda,C_r)\) and
establishing an exact rate of convergence
of the risk of the estimator \(\hat\vartheta_r\)  requires sharp probabilistic 
bounds similarly to the convex case and is beyond the scope of the present paper.
In order to compute \(\hat\vartheta_r\) in practice, it is crucial to know the radius \(r\).
It is clear that \(\Cb_{r_2} \subset \Cb_{r_1}  \) for all \(r_2 > r_1 > 0\). 
Define the true radius \(r^\star\) corresponding to a set \(C\) as
\begin{EQA}[c]
	r^\star = \sup \{r> 0: C \in \Cb_r \}\,.
\end{EQA}
It is intuitively evident that it is better to underestimate the radius \(r^\star\), because \(\Cb_{r^\star} \subset \Cb_{r} \)
for all \(r^\star > r > 0\). 
On the other hand, there is a typical bias-variance tradeoff in choosing the optimal parameter \(r\). 
We prefer to use large values of \(r\) when the number of sample points is scarce, 
and can afford to use small values  \(r < r^\star\) when there is an abundance of the sample points. 
We here propose a procedure for estimating \(r^\star\) based on Lepski's method \cite{lepskii1992asymptotically}, 
see also a more accessible reference, Section 8.2.1 in \cite{giné2015mathematical}.
Fix some \(R \in \R_{+}\) and \(K \in \N \) and break the interval \((0,R)\) down into \(K + 1\) pieces \(0 < r_1 < ... < r_K < R\)  
of equal length. Define an estimator for \(\hat{r}\) as 
\begin{EQA}[c]
	\label{adaptR}
	\hat{r} := \inf \big\{ r_{k-1} |\, \exists   k^\prime \le k\, : \,  | \hv_{r_k} -  \hv_{r_{k^\prime}} | >  \kappa_n \big\} \wedge r_K \,,
\end{EQA}
with \(\kappa_n = N_\delta/ n^{2}\). This calibration is suggested by Theorem~\ref{UpBoundOracle} in view of Corollary~\ref{EfronsIdenPoisson}. 
The asymptotic behaviour of the estimator \(\hat\vartheta_{\hat{r}}\) depends on an exact deviation inequality on the missing volume \(|C_r \setminus \Ch_r|\).
This question however is beyond the scope of the present paper. We provide a numerical study of this 
adaptation procedure in Section~\ref{SimulationsSection}.

\subsection{Compact sets}
Interestingly the class of all compact sets \(\Kb\) of non-zero Lebesgue measure satisfies Assumption~\ref{Assumption} as well. 
The richness of this class makes it foremost for conducting statistical inference, yet 
very little has been proposed and studied so far.
Estimation of this class of sets was studied in \cite{devroye1980detection}, where it was shown that 
the union of small Euclidean balls centred at the points of the sample is a consistent estimator of a compact set.
The \(\Kb\)-wrapping hull is just the union of sample points and so \(N_\partial = N\) and \(|\hat{K}| = 0\) a.s.
Hence for the oracle estimator in \eqref{OracleEst} we have
\begin{EQA}[c]
	\hv_{\Kb,oracle} := \frac{N}{\lambda}\,.
\end{EQA}
This estimator is unbiased 
and from \eqref{RiskNaiveEst} the following result immediately follows.

\begin{lemma}For known intensity $\lambda > 0$, 
	the worst case mean squared error of the oracle estimator \(\hv_{\Kb,oracle}\) over the parameter class $\Kb$ decays as $\lambda\uparrow \infty$  like $\lambda^{-1}$:
	\begin{EQA}[c]
		\sup_{K \in \Kb} \frac{1}{|K|} \E \big[(\hv_{\Kb,oracle} - |K|)^2\big]  =  \frac{1}{\lambda}\,.
	\end{EQA}
\end{lemma}
It seems impossible without imposing further structure on the class \(\Kb\) to estimate \(\lambda\) in this scenario.

\subsection{Polytopes}
It was noted in \cite{BaldinReiss16}, that the estimator of the volume based on the 
\(\Cb_{\text{conv}}\)-wrapping hull estimator (the convex hull \(\Ch\)) is adaptive to the class of polytopes \(\Pb\) (see Remark~3.3). 
In fact, the estimator 
\begin{EQA}[c]
	\label{polytopes_hull_estimator}
	\hv_{\Pb,oracle} := \frac{N_\partial}{\lambda} + |\Ch|\,
\end{EQA}
satisfies
\begin{lemma}
	\label{PolyLemR}For known intensity $\lambda > 0$, 
	the worst case mean squared error of the oracle estimator \(\hv_{\Pb,oracle}\) over the parameter class $\Pb$ decays as $\lambda\uparrow \infty$  like $\lambda^{-2}(\log(\lambda))^{d-1}$:
	\begin{EQA}[c]
		\limsup_{\lambda\to \infty} \lambda^{2}(\log(\lambda))^{1-d} \sup_{P \in \Pb, |P| > 0} \Big\{\E\big[(\hv_{\Pb,oracle} -\abs{P})^2\big]\Big\}  < \infty\,.
	\end{EQA}
\end{lemma}
We stress here, however, that the class polytopes \(\Pb\) does not satisfy Assumption~\ref{Assumption}. The framework 
applies to the class of convex sets \(\Cb\) and Lemma~\ref{PolyLemR} only allows to improve the rate for the subclass of \(\Cb\).
The class \(\Pb\) is stable only under finite intersections; taking arbitrary (possibly uncountable) intersections,
one can obtain an element not lying in the class.

\subsection{Polytopes with fixed directions of outer unit normal vectors}
\label{PolFixeDirSec}
The class of polytopes \(\Pb_{\Sb_k} \) with fixed directions \(\Sb_k = \{u_1,...,u_k\}\)
of outer unit normal vectors \(u_k \) belonging to the unit sphere \(\mathbb{S}^{d-1}\) provides another interesting example  of intersection stable sets. We assume the class is well-defined in the sense that there exists a polytope \(P_{\Sb_k}\) whose outer unit normal vectors  are  exactly \(\{u_1,...,u_k\}\). Without loss of generality we may assume \(\Eb =P_{\Sb_k} \).
The \(\Pb_{\Sb_k}\)-wrapping hull \(\Ph \) is a polytope with at most 
\(k\) facets and is given by 
\begin{EQA}[c]
	\Ph := \bigcap_{P  \in \Pb_{\Sb_k}: \{X_1,...,X_N\} \in P} P\,.
\end{EQA} 
The oracle estimator and the data-driven estimator are thus defined as
\begin{EQA}[c]
	\label{polytopes_fixed}
	\hv_{\Pb_{\Sb_k} ,oracle} := \frac{N_\partial}{\lambda} + |\Ph|\,,
	\quad 
	\hv_{\Pb_{\Sb_k} }  = \frac{N + 1}{N_\circ + 1} |\Ph|\,,
\end{EQA}
where the number of points lying on the boundary of the wrapping hull \(N_\partial\)
is equal to the number of facets of the wrapping hull and hence upper-bounded by \(k\).
According to the general scheme, the rate of convergence of the risk for the oracle estimator 
\(	\hv_{\Pb_{\Sb_k} ,oracle} \)
and the final estimator  \(	\hv_{\Pb_{\Sb_k} } \) rests upon a deviation inequality for the missing volume and is established in the following theorem which is proved in the Appendix.

\begin{theorem}
	\label{polytopesFixed}
	The worst case mean squared error of the  estimator  \(	\hv_{\Pb_{\Sb_k} } \) over the parameter class $\Pb_{\Sb_k} $ satisfies:
	\begin{EQA}[c]
		\sup_{ P \in \Pb_{\Sb_k}} \E\big[( \hv_{\Pb_{\Sb_k} } - |P|)^2 \big]
		\lesssim \frac{kW(\lambda/k) }{\lambda^2}\,,
	\end{EQA}
	whenever \(\lambda |P| \ge 1\), where \(W\) is the Lambert-W function that satisfies \(W(z\ex^{z}) = z\). The rate can further 
	be upper-bounded by \(k \log(\lambda / k )/\lambda^2\).
\end{theorem}

\section{Illustrative simulations}
\label{SimulationsSection}
In this section, we  illustrate the performance of the proposed estimators for a class of \(r\)-convex sets.
Our primal example of an \(r\)-convex set for simulations is the annulus \(C_{r^\star} = B(0.5,0.5)\setminus B(0.5,0.25)\). Thus clearly \(C_{r^\star} \in \Cb_r \) for all \(0 < r \le 0.25\). Figure~\ref{r-convex_alpha_varies} depicts the \(r\)-convex hull estimator \eqref{r_convex_hull_estimator} for \(r = 0.01, 0.04,0.2,1\) based on the observations of the PPP with \( \lambda  = 300\).
An important observation is that once the value of \(r\) is larger than the true radius \(r^\star\)
of an \(r\)-convex set, the \(r\)-convex hull essentially misses the ``holes'' of radius \(r^\star\). One should bear this in mind
when using large values of \(r\)  for constructing the oracle estimator when the number of observation points is small. This subtle issue is depicted in Figure~\ref{r_convex_fixe_alpha}, where 
the root mean squared error of the oracle estimator \(\hv_{r,oracle}\) for the volume  (black line), based on the \(r\)-convex hull with \(r = 0.04\), converges to the area of the ``hole'' of size \(\pi (0.25)^2 \approx  0.196\). Another striking point is that when the number of observations is small, the 
\(r\)-convex hull with a small value of \(r\) essentially coincides with the points themselves and thus the RMSE of the oracle estimator 
\(\hv_{r,oracle}\) coincides with the RMSE of \(\hv_{\Kb,oracle}\) and equals \(|C_{r^\star}|/\lambda \) (red line in Figure~\ref{r_convex_fixe_alpha})! Finally we depict RMSE estimates for
the oracle estimator \(\hv_{r,oracle}\) for different \(r\) in Figure~\ref{r_convex_n_varies_alpha_fixed}. One can clearly see the regions of decreasing 
value of the RMSE, the fairly flat value of RMSE and the jump when \(r\) becomes larger than the true parameter \(r^\star\).
Table~\ref{table1} further collects the Monte Carlo estimates  of the number of points  \(N_\circ\) lying inside the wrapping hull,
the number of points \(N_\partial\)  on the boundary of the wrapping hull and the number of 
isolated points \(N_{iso}\) of the boundary of the wrapping hull. For analyzing the performance 
of the  
adaptive estimator proposed in Section~\ref{RConnSect}, we break  the interval \([0.06, 0.5]\) into pieces of length \(0.02\), 
compute the estimates of the radius \(\hat{r}\) from \eqref{adaptR} and the estimates of the RMSE of \(\hv_{\hat{r}}\) 
based on 200 Monte Carlo iterations in Table~\ref{table2}.

\begin{figure}[tp]
	\centering
	\includegraphics[width=0.8\linewidth]{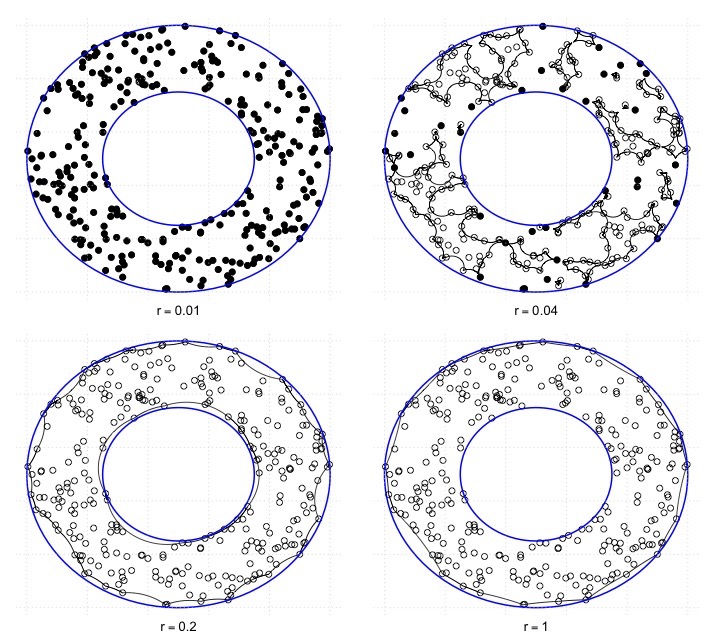}\\[-5mm]
	\caption[]{The four weakly-convex set with \(r^\star = 0.25\) (blue), observations of the PPP with \(\lambda = 300\) (points)  and  their \(r\)-convex hulls for different values of \(r\) (black).}
	\label{r-convex_alpha_varies}
\end{figure} 

\begin{figure}[tp]
	\centering
	\includegraphics[width=0.7\linewidth]{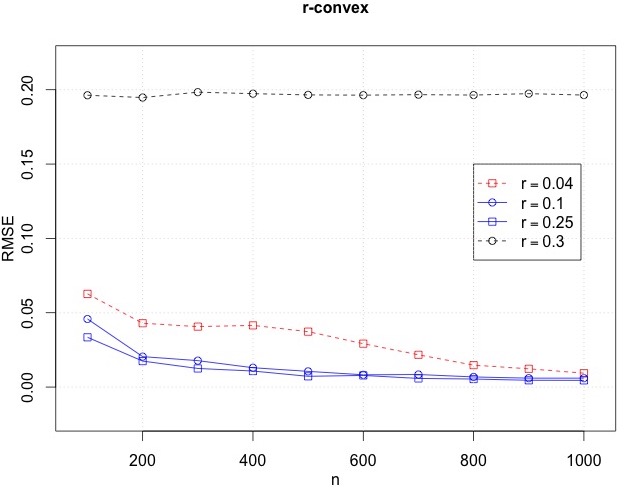}\\[-5mm]
	\caption[]{Monte Carlo RMSE estimates for the oracle estimator for the volume of the annulus \(B(0.5,0.5)\setminus B(0.5,0.25)\) with respect to the sample size.}
	\label{r_convex_fixe_alpha}
\end{figure} 

\begin{figure}[tp]
	\centering
	\includegraphics[width=0.7\linewidth]{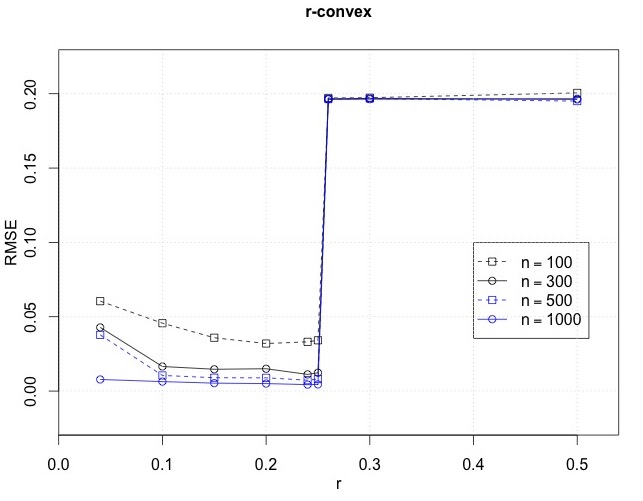}\\[-5mm]
	\caption[]{Monte Carlo RMSE estimates for the oracle estimator for the volume of the annulus \(B(0.5,0.5)\setminus B(0.5,0.25)\) with respect to \(r\).}
	\label{r_convex_n_varies_alpha_fixed}
\end{figure} 

\begin{table}[tp]
	\footnotesize
	\resizebox{\linewidth}{!}{%
		\begin{tabular}{c  c  c  c  c c c  }
			\multicolumn{7}{c}{\(r = 0.04\) }\\
			\(n = \lambda / |A|\)  & \(N_\circ\)& \(N_\partial\) & \(N_{iso}\) & \( RMSE(\hv_{r,oracle})\) & \( RMSE(\hv_{r})\) & 
			\( \frac{RMSE(\hv_{r}) }{RMSE(\hv_{r,oracle})} \)   \\
			50 & 0.13 & 49.8 & 44 &0.087 &0.55 &  6.31 \\
			100 & 1.5 &  98.37& 66 &0.059& 0.33 & 4.65 \\
			200 & 23.6 &  175.8& 58 & 0.042& 0.15 & 3.7 \\
			300 & 90.4 &  211 & 33 &0.039 &0.109 & 2.82 \\
			400 & 191.6 & 210 & 17 & 0.043 &0.085 &  1.97\\
	\end{tabular}}\vspace{10pt}
	\resizebox{\linewidth}{!}{
		\begin{tabular}{c  c  c  c  c c c }
			\multicolumn{7}{c}{\(r = 0.1\) }\\
			\(n = \lambda / |A|\)  & \(N_\circ\)& \(N_\partial\) & \(N_{iso}\) & \( RMSE(\hv_{r,oracle})\) & \( RMSE(\hv_{r})\) & 
			\(  \frac{RMSE(\hv_{r}) }{RMSE(\hv_{r,oracle})} \)  \\
			50 & 8.5 & 42.16 &8.52   & 0.071 &0.138 & 1.94\\ 
			100 & 45.9 & 53.34 & 1.37 & 0.043 &0.064&1.48  \\
			200 &  138.2 & 60.95 &0.03  &0.021  &0.027& 1.26 \\
			300 &  233.4 & 68.08 &0  &  0.015&  0.018&1.20\\
			400 & 326.1 &  74.40& 0& 0.013 & 0.015 & 1.15\\
	\end{tabular}}\vspace{10pt}
	\resizebox{\linewidth}{!}{
		\begin{tabular}{c  c  c  c  c c c  }
			\multicolumn{7}{c}{\(r = 0.25\) }\\
			\(n = \lambda / |A|\)  & \(N_\circ\)& \(N_\partial\) & \(N_{iso}\) & \( RMSE(\hv_{r,oracle})\) & \( RMSE(\hv_{r})\) & 
			\(  \frac{RMSE(\hv_{r}) }{RMSE(\hv_{r,oracle})} \)   \\
			50 & 24.75 & 24.21 & 0.06 & 0.061 &0.085 &1.39 \\
			100 &68.58  & 29.60 & 0 &0.033  & 0.0405& 1.20\\
			200 & 163.75 &36.03  &  0&0.018  & 0.019 &1.04\\
			300 & 261.44 & 40.68 &  0&0.0108  & 0.0124 &1.13\\
			400 & 357.41 &  44.17&  0&  0.0096 &  0.0104 &1.076\\
	\end{tabular}}\vspace{10pt}
	\resizebox{\linewidth}{!}{
		\begin{tabular}{c  c  c  c  c c c  }
			\multicolumn{7}{c}{\(r = 0.3\) }\\
			\(n = \lambda / |A|\)  & \(N_\circ\)& \(N_\partial\) & \(N_{iso}\) & \( RMSE(\hv_{r,oracle})\) & \( RMSE(\hv_{r})\) & 
			\( \frac{RMSE(\hv_{r}) }{RMSE(\hv_{r,oracle})}  \)  \\
			50 &  30.71 &18.70  &  0& 0.208 &0.340  &1.628\\
			100 &77.59  &23.26  & 0 & 0.2002 & 0.258&1.29 \\
			200 &  170.30& 29.39 & 0 & 0.1982 &0.232 &1.17 \\
			300 & 265.17 & 33.89 & 0 & 0.1978 & 0.223& 1.13 \\
			400 &362.43  &  37.89& 0 &0.1987 &0.219 &  1.10\\
	\end{tabular}}
	
	\caption{Monte Carlo RMSE estimates for the oracle estimator \(\hv_{r,oracle}\) and for the fully data-driven
		estimator \(\hv_r\) for the volume of the annulus \(A = B(0.5,0.5)\setminus B(0.5,0.25)\) 
		with respect to \(r\) and \(n = \lambda |A|\), the number of points lying inside the wrapping hull \(N_\circ\),
		the number of points on the boundary of the wrapping hull \(N_\partial\) and the number of 
		isolated points of the boundary of the wrapping hull \(N_{iso}\).}
	\label{table1}
\end{table}

\begin{table}[tp]
	\footnotesize
	\begin{center}
		\begin{tabular}{c  c  c   }
			\(n = \lambda / |A|\)  & \(\hat{r} \) & \( {RMSE(\hv_{\hat{r}})}  \) \\
			50 & 0.088 & 0.36  \\
			100 & 0.085 &  0.160 \\
			200 & 0.084 &  0.069 \\
			300 & 0.105 &  0.033  \\
			400 & 0.125 & 0.0182 \\
			500 & 0.149 & 0.0123 \\
			1000 & 0.165 & 0.0056 \\
		\end{tabular}
	\end{center}
	\caption{Monte Carlo RMSE estimates for the adaptive estimator \(\hv_{\hat{r}}\) 
		for the volume of the annulus \(A = B(0.5,0.5)\setminus B(0.5,0.25)\) 
		with respect to \(n = \lambda |A|\).}
	\label{table2}
\end{table}

\subsection{Appendix}
\subsubsection{Proof of Theorem~\ref{SuffComplThm}}
Sufficiency follows from the Neyman factorisation criterion applied to the likelihood function \eqref{fPPeb},
while completeness follows by definition provided that we show 
\begin{EQA}[c]
	\forall A \in \Ab: \E_A\bigl[T(\Ah)\bigr] = 0 \implies  T(\Ah) = 0 \;\;\; \P_{\mathbf E}-a.s.
\end{EQA}
for any \(\Ac\)-measurable function \(T: \Ab \to \R\).
From the likelihood  in \eqref{fPPeb} for $\lambda=\lambda_0$, we derive
\begin{EQA}
	\E_A\bigl[T(\Ah)\bigr] & = &\E_{\mathbf E}\Bigl[T(\Ah) \exp\bigl(\lambda\abs{{\mathbf E}\setminus A}\bigl){\bf 1}(\Ah \subseteq A )\Bigr] \,.
\end{EQA}
Since \(\exp(\lambda\abs{{\mathbf E}\setminus A})\) is deterministic,  we have \(\forall A \in \Ab\)
\begin{EQA}[c]
	\E_A\bigl[T(\Ah)\bigr] = 0 \implies  
	\E_{\mathbf E}\bigl[T(\Ah){\bf 1}(\Ah \subseteq A) \bigr] = 0\,. \,
\end{EQA}
Splitting \(T = T^{+} - T^{-}\)
with non-negative \(\Ac\)-measurable functions \(T^{+}\) and \(T^{-}\), we infer that the measures
\(\mu^\pm(B)=\E_{\mathbf E}[T^{\pm}(\Ah){\bf 1}(\Ah \in B)] \), $B\in \Ac$, agree on \(\{[B] \,|\, B \in \Ab \}\), where 
\([B] = \{A\in \Ab |  A \subseteq B \}\). Since the brackets \(\{[B] \,|\, B \in \Ab \}\) generate the \(\sigma\)-algebra \(\Ac\) the measures 
\(\mu^\pm(B)\) agree on all sets in \(\Ac\), in particular on  \(\{T > 0\} \) and \(\{T < 0\} \), which entails $\E_{\mathbf E}[T^+(\Ah)]=\E_{\mathbf E}[T^-(\Ah)]=0$.
Thus, \(T(\Ah) = 0\) holds  \(\P_{\mathbf E}\)-a.s.

\subsubsection{Proof of Theorem~\ref{polytopesFixed}}
Let us denote by \(\rho_1(A,B) = |A\triangle B|\)  the symmetric distance between 
two compact subsets \(A\) and \(B\) of the compact convex set \(\Eb\) in \(\R^d\). 
Recall that an \emph{\(\eps\)-net} of the class \(\Pb_{\Sb_k}\) with respect to the metric \(\rho_1\) is a collection \(\{P^1,...,P^{N_\eps} \} \in \Pb_{\Sb_k }\)
such that for each \(P \in \Pb_{\Sb_k}\), there exists \(i \in \{1,...,N_\eps\} \) such that \(\rho_1(P,P^i) \le \eps\). The \emph{\(\eps\)-covering 
	number} \(N(\Pb_{\Sb_k}, \rho_1, \eps)\) is the cardinality of the smallest \(\eps\)-net. The \emph{\(\eps\)-entropy} of the class \(\Pb_{\Sb_k}\) is 
defined by \(H(\Pb_{\Sb_k}, \rho_1, \eps) = \log_2 N(\Pb_{\Sb_k}, \rho_1, \eps)\). Furthermore, it follows by dilation of a set that for  \(\Ph \in \Pb_{\Sb_k}\) there exists \(\hat{m} \in \{1,...,N_\eps\} \) such that   \(\Ph \subseteq P^{\hat{m}} \subseteq P  \) and \(\rho_1(\hat{P}, P^{\hat{m}}) \le c \eps\) for some universal constant \(c > 1\) and \( \eps\) small enough. We thus obtain for all \(P \in \Pb_{\Sb_k}\) and \(x > 0\),
\begin{EQA}
	\P\big(|P \setminus \Ph| > x/\lambda + 2 c \eps \big) & \le& 	\P\big(|P \setminus P^{\hat{m}} | > x/\lambda + c \eps \big) 
	\\
	&\le&  \sum_{m: |P \setminus P^{m} | > x/\lambda +  c \eps } 
	\P\big(\Nc (P \setminus P^{m}) = 0 \big)  \\
	&\le& \exp\big( -x - c\lambda \eps + H(\Pb_{\Sb_k}, \rho_1, \eps)  \big) = \ex^{-x}\,,
\end{EQA}
plugging in \(\eps\)
that solves \(H(\Pb_{\Sb_k}, \rho_1, \eps) = c \lambda \eps \). The \(\eps\)-covering 
number \(N(\Pb_{\Sb_k}, \rho_1, \eps)\)  of the class \(\Pb_{\Sb_k}\) can  be bounded by \((C/\eps)^{k}\) for some universal constant \(C > 1\). As a result, 
the asymptotic rate follows using  Fubini's theorem combined with Theorem~\ref{UpBoundOracle} and Theorem~\ref{new_estimator_risk_wrap}.

\section*{Acknowledgements}
The author is  grateful to Markus Rei{\ss} at 
Humboldt-Universit{\"a}t zu Berlin for hosting him in 
January 2016 and  valuable discussions which originated this research. The author further thanks Quentin Berthet for helpful comments.

%

%
%
%







\bibliographystyle{economet}



\bibliography{ref}

\end{document}